\documentclass{article}
\usepackage[utf8]{inputenc}
\usepackage{graphicx}
\usepackage{amsfonts}
\usepackage{amsmath}
\usepackage{nccmath}
\usepackage{amssymb}
\usepackage{fancyhdr}
\usepackage{amsthm}
\usepackage{array}
\usepackage{systeme}
\usepackage{xcolor}
\usepackage{mathtools}
\DeclarePairedDelimiter{\evdel}{\langle}{\rangle}
\newcommand{\ev}{\operatorname{E}\evdel}
\usepackage{hyperref}

\title{The discrete renewal theorem with bounded interevent times}
\author{Rohan Shenoy}
\date{November 2021}

\begin{document}

\maketitle

The purpose of this note is to prove the celebrated Discrete Renewal Theorem in a common special case, using only very elementary methods from real analysis, rather than markov chain theory, complex analysis, or generating functions.

\section{Probabilistic Sequence}
To introduce the problem, consider a class of board games in which a player's counter makes a sequence of moves in a fixed direction along a line of squares $S_n,n\geq0$. The counter starts from $S_0$, with the sizes of successive moves determined by the roll of a die (or multiple dice), which may be biased. 

For the $n^{th}$ square $S_n$, it is natural to ask for the probability that the counter ever lands on $S_n$, denoted $u_n$. This is especially valuable where $S_n$ is a square on which the player gains some reward or pays some penalty. Note that $u_0=1$ and we have allowed the line of squares to be semi-infinite. 

By definition, the length $X$ of any jump of the counter is independent of all other jumps, and we denote its probability distribution by
    $$f_j=P(X=j)\hspace{0.2cm},\hspace{0.2cm}j\geq1$$
with the sum of $f_j$ being $1$. Example: If jumps are given by a fair cubical die, then
    $$f_j=\frac{1}{6}\hspace{0.2cm},\hspace{0.2cm}1\leq j\leq6$$
and we require the following assumption
\newtheorem{Assumption}{Assumption}
\begin{Assumption}\label{asm}
    For some finite $S$, $f_j=0$ for $j>S$.
\end{Assumption}
Under Assumption (\ref{asm}), we shall prove by elementary methods in Theorem 2 below, that
    \begin{equation}\label{ev}
        \lim_{n\to\infty}u_n=\frac{1}{\ev{X}}
    \end{equation}
where $\ev{X}$ is the expected value of any jump $X$. That is to say, for all practical purposes, sufficiently distant points are all equally likely to be visited by the counter with probability $\frac{1}{\ev{X}}$. First we give this preliminary
    \newtheorem*{Thrm1}{Theorem 1}
        \begin{Thrm1}
            For all $n\geq1$
            \begin{equation*}
                u_n = f_nu_0 + f_{n-1}u_1 + ... + f_1u_{n-1},\tag{a}
            \end{equation*}
            and
            \begin{equation*}
                1 = P(X>n) + P(X>n-1)u_1 + ... + P(X>1)u_{n-1} + u_n,\tag{b}
            \end{equation*}
        \end{Thrm1}
    \begin{proof}
        (a) For $1\leq j \leq n$, the counter leaves $S_0$, lands first on $S_{n-j}$, and then subsequently lands on $S_n$ with probability $f_ju_{n-j}$. Summing these probabilities yields our result, using the partition theorem (the law of total probability) \cite{feller1957introduction}.
        
        (b) There are 3 cases to consider
        \begin{itemize}
            \item The counter `visits' $S_n$ with probability $u_n$
            \item For some $j<n$, the counter first visits $S_j$ and then makes a jump greater than $n - j$ with probability $u_jP(X>n-j)$
            \item The counter jumps from $S_0$ directly `over' $S_n$, with probability $P(X>n)$
        \end{itemize} Summing all of these probabilities, the partition theorem \cite{feller1957introduction} is again used to yield our result.
    \end{proof}
We make the following observations
\begin{enumerate}
    \item A positive sequence $u_n$ defined by (a) is called a \textit{renewal sequence} with respect to the distribution $f_n$.
    \item The sum on RHS of (a) is called the \textit{convolution} of the sequences $u_n$ and $f_n$ \cite{cox1983point}.
    \item Under Assumption (\ref{asm}) there are at most S and S+1 terms in the RHSs of (a) and (b) respectively.
\end{enumerate}

\section{Convergence of Sequence}
To justifiably use the limit formula in (\ref{ev}), we need to establish that $u_n$ converges as $n\to\infty$
\newtheorem*{Convergence}{Theorem 2}
\begin{Convergence}
    $u_n$ has a finite limit as $n\to\infty$, this being
    $$\lim_{n\to\infty}u_n=\frac{1}{\ev{X}}$$
    where $\ev{X}$ is the expected value of any jump $X$
\end{Convergence}
\begin{proof}
From Assumption (\ref{asm}), for some finite $S$, $f_j=0$ for $j>S$. For $n>S$, Theorem 1(a) then reduces to
    $$u_n=f_{S}u_{n-S} + f_{S-1}u_{n-S+1} + ... + f_2u_{n-2} + f_1u_{n-1}$$
Letting the maximum of the terms $(u_{n-S},...,u_{n-1})$ be denoted $M_n$, all of these terms are clearly less than or equal to $M_n$. As the sum of $f_j$ is 1, $u_n$ must be less than or equal to $M_n$. Avoiding the trivial case where $u_n$ is constant, it can be proven by induction that at least one of $(u_{n-S},...u_{n-1})$ is not $M_n$ and the inequality becomes strict,
    \begin{equation}\label{<<}
        u_n<M_n
    \end{equation}
and so all terms ($u_{n-S+1},...,u_n$) are all less than or equal to $M_n$ such that
    $$M_{n+1}\leq M_n$$
and hence $M_n$ is a monotonically decreasing sequence. Similarly denoting $m_n$ the minimum of $(u_{n-S},...,u_{n-1})$, with the precise same logic, $m_n$ is a monotonically increasing sequence. Clearly with the context of probability, both sequences are bounded. By the Monotone Convergence Theorem \cite{MCT}, $M_n$ and $m_n$ both converge to finite limits as $n\to\infty$. We wish to prove that these limits are the same.

Denoting the limit of $M_n$ and $m_n$ as $M$ and $m$ respectively, we assume, for the sake of contradiction, that $m\neq M$. There must then exist a $\delta>0$ such that
    $$M-m=k\delta$$
Where $kf_j>1\hspace{0.2cm},\hspace{0.2cm}1\leq j\leq S$. As $m_n$ is monotonically increasing to $m$, for all $n$ it is less than or equal to $m$. We have
    \begin{equation}\label{m_n}
        m_n\leq M-k\delta.
    \end{equation}
As $M_n$ is monotonically decreasing to $M$, there must be some $N$ such that for all $t>N$, 
    \begin{equation}\label{M_n}
        M_t-M<\delta.
    \end{equation}
Now consider equation (a) under Assumption (\ref{asm}). One of $(u_{t-S},...,u_{t-1})$ is $m_t$ and all others are less than or equal to $M_t$. Let $m_t$ here have coefficient $f'$, then as the sum of $f_j$ is 1, 
    $$u_t\leq f' m_t+(1-f')M_t$$
and we can substitute equations (\ref{m_n}) and (\ref{M_n}) into this
    $$u_t\leq M-(kf'+f'-1)\delta$$
With all the above constraints, $(kf'+f'-1)\delta$ is positive such that for all $t>N$, 
    $$u_t<M$$
But then all terms $(u_{t},...,u_{t+S-1})$ are strictly less than $M$. Setting $n=t+S$ we have
    $$M_n<M$$
A clear contradiction of the proven fact that $M_n$ monotonically decreases to $M$. The assumption that the sequences have different limits is false, and so $M_n$ and $m_n$ converge to the same limit. From equation (\ref{<<}),
    $$m_n<u_n<M_n.$$
As $M_n$ and $m_n$ have the same limit, the Squeeze Theorem \cite{fuller1977sequences} applies and $u_n$ must have this same limit $L$.

We may now complete the proof of Theorem 2. We have established that $u_n$ converges to some limit $L$ as $n\to\infty$. Allowing $n\to\infty$ in Theorem 1(b) yields the required result immediately, using the fact that the sum is finite, noting that $X$ has a proper distribution [so $P(X > 0) = 1 $], and using the tail sum formula for $\ev{X}$. That is, for a positive integer-valued X, 
    $$\ev{X} = P(X > 0) + P(X > 1) + P(X > 2) + ... $$
The result follows
    $$\lim_{n\to\infty}u_n=\frac{1}{\ev{X}}$$
\end{proof}
We note, that in contrast to the Erdös-Feller-Pollard theorem \cite{x} (described below) which proves the more general renewal theorem, we have been able to prove this specific result using only very elementary methods.

\section{Background and History}
Theorem 2, as given above, is actually true without Assumption (\ref{asm}), where the limit becomes zero when $\ev{X}$ is infinite. This was shown in a famous paper of 1949 by P. Erdös, W. Feller, and H. Pollard \cite{x}, using the discrete generating function [\textit{dgf}], and proved a key property of power series with positive coefficients to acquire the result. Explicitly, if the sequences $u_n$ and $f_n$ have \textit{dgf}s U(s) and F(s) respectively, then from Theorem 1(a), we have 
    \begin{equation}\label{poln}
        U(s) = \frac{1}{1 - F(s)}, 
    \end{equation}
as used in the Erdös-Feller-Pollard theorem. For a simpler proof of this theorem, a summary of all the standard notation, and its various applications, one can see chapter XIII of W. Feller's book: `An Introduction to Probability Theory and its Applications', Volume I, 3rd edition (1968). One can also find the original proof here \href{https://old.renyi.hu/~p_erdos/1949-01.pdf} (renyi.hu) \cite{x}.

Note that under Assumption (\ref{asm}), $[1 - F(s)]$ is a polynomial, and it follows from (\ref{poln}), and the theory of partial fractions, that
    $$\lim_{n\to\infty}u_n=\frac{1}{E[X]}$$
which is our Theorem 2. However, in the special case of Assumption (\ref{asm}), the elementary proof outlined in sections $1$ and $2$ is sufficient.

Discrete renewal theory is very strongly linked with certain properties of Markov chains. The discrete renewal theorem can be proved using suitable Markov chain convergence theorems, and vice versa \cite{smith1958renewal}.

In more general renewal processes, the lengths of jumps $X$ are allowed to have an arbitrary distribution on the positive real line. In such cases, the Laplace transform may be used in place of the \textit{dgf}, and a suitably modified version of the renewal theorem proved. Even more generally, one may allow counters to jump both forwards and backwards, so that $X$ may also take negative values with a distribution on the entire real line. Here the Fourier transform may be used to acquire appropriate results \cite{resnick1992adventures}.

\begin{center}
    \includegraphics[width=0.71cm]{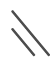}
\end{center}

\section*{Acknowledgements}
I would like to acknowledge the help of an anonymous referee for providing their detailed report. I would also like to acknowledge the support of my supervisor Mr Stuart Andrew for comments on the paper.

\bibliographystyle{unsrt}
\bibliography{Bibliography.bib}

\begin{flushright}
    ROHAN MANOJKUMAR SHENOY\\
    Nottingham High School, Waverley Mount, Nottingham, NG7 4ED
    Shenoy.r.m@nottinghamhigh.co.uk
\end{flushright}

\end{document}